\documentclass{article}
\usepackage{graphicx,amsmath, amssymb, amsthm,url,comment} 
\usepackage[margin=1in]{geometry}
\newtheorem{thm}{Theorem}
\newtheorem{defn}{Definition}
\newtheorem{lemma}{Lemma}
\newtheorem{cor}{Corollary}
\usepackage{graphicx} 
\newcommand{\A}{{ \mathcal{A} }}
\newcommand{\BM}{{ \mathfrak{A} }}
\newcommand{\Adj}{\{ A_0, \dots, A_d \}}

\newcommand{\rl}[1]{{ \overset{#1}{\sim} }}
\newcommand{\Rel}{ \mathcal{R} }

\title{Antipodal SRG covers}
\author{Patrick Cesarz \\ Department of Mathematics \\ Alvernia University \\ 400 St. Bernardine Street \\ Reading, PA 19607 \\ email: patrick.cesarz@gmail.com \\ \\ Jason Williford \\ Department of Mathematics and Statistics \\ University of Wyoming \\ 1000 E. University Ave. \\ Laramie, WY 82071 \\ email: jwillif1@uwyo.edu}

\begin{document}

\title{Four class association scheme double covers of strongly regular graphs}

\maketitle

\begin{abstract}
    In this paper we give necessary conditions for 4-class association schemes that are generated by double covers of strongly regular graphs. These conditions are applied to open cases for diameter 4 antipodal distance-regular graphs. Using these conditions we are able to show the nonexistence of four cases in the table of Brouwer, Cohen and Neumeier: $\{20,18,3,1;1,3,18,20\}$, $\{22,21,3,1;1,3,21,22\}$, $\{ 54,50,5,1;,1,5,50,54 \}$, $\{ 170,162,9,1;,1,9,162,170 \}$.
    More generally, we show there is no distance-regular graph with intersection array $\{k,b_1,b_2,1;1,b_2,b_1,k \}$ where $b_2 \neq 1$ and $\frac{ k b_1 }{4}(1+k+\frac{ k b_1 }{2b_2})$ is odd. Tables are also given for more general 4-class association schemes generated by double covers of strongly regular graphs. 
\end{abstract}

\section{Introduction}

The notion of a double cover is very useful in many areas of mathematics. Indeed, in algebraic topology and algebraic geometry their use goes back to the 19th century. In combinatorics, double covers of graphs have had many applications in recent years. These are often in the form of a signing of the edges of a known graph with plus/minus signs. For example, in \cite{ramanujan} the authors use successive double covers to construct bipartite Ramanujan graphs of all possible valencies. In 2019, Huang \cite{sensitivity} proved the sensitivity conjecture using a signing of the edges of a hypercube. This signing describes the unique $C_4$-free double cover of the cube, discovered by Cohen and Tits \cite{ncube}. 

Double covers are important in association schemes, since they arise  as the graphs of  imprimitive schemes with a perfect matching as a relation. For example, double covers of certain strongly regular graphs and symplectic dual polar space graphs were used to make new association schemes with the $Q$-polynomial property (see \cite{PW}, \cite{MW}).

In finite geometry, Moorhouse \cite{M1} used double covers of semibiplanes to construct new projective planes by associating these covers with a linear system over $\mathbb{F}_2$. This linear system is the inspiration for the present work. In this paper, we will to derive new conditions of 4-class association schemes that are generated by double covers of strongly regular graphs. These will be applied to hypothetical antipodal distance-regular graphs of diameter 4. 

%













\section{ Graph Double Covers }

Let $\Gamma$ be a graph with a vertex partition $\pi$ into independent sets of size 2, which we shall refer to as \textit{fibers}. Suppose further that any two fibers either have no edges between them or a perfect matching. Define $\Gamma/\pi$ to be the \textit{quotient graph} with respect to this partition, i.e. the graph with fibers as vertices, with an edge between fibers precisely when they are joined by a matching in $\Gamma$. We say that $\Gamma$ is a \textit{double cover} of the graph $\Gamma/\pi$. 


While it is convenient to define the quotient graph as derived above from the double cover, one is often more interested in the inverse problem: given a graph, construct a double cover with certain properties. From this point onward, we will use the notation $\Gamma$ for our graph and $\hat{\Gamma}$ for a putative double cover of $\Gamma$. For each vertex $x$ of $\hat{\Gamma}$, the other vertex $y$ such that $\{x,y\}$ is a fiber is called the \textit{antipode} of $x$. 





Given a graph, it is natural to ask how one would find all double covers of this graph.  
A standard way of encoding a double cover of a graph $\Gamma$ is by signing all of the edges of $\Gamma$ with $+,-$ signs. We will denote the signing function by $\sigma : E \to \{-,+\}$, and refer to the pair $\Gamma,\sigma$ as a signed graph. Any such signing gives a double cover with the vertex set $\hat{V} = \{ (x,i) : x \in V, i \in \{1,2\} \}$, with $(x,i) \sim (y,i)$ iff $x \sim y$ with $\sigma(\{x,y\})=+$, and $(x,i) \sim (y,j)$ with $i \neq j$ iff $x \sim y$ with the edge $\sigma(\{x,y\})=-$. Conversely, any double cover can be described (up to isomorphism) by such a signing of the edges of $\Gamma$. 

The signing referred to above is not unique, different signings of the graph may or may not give isomorphic double covers. For example, picking a vertex $x$ and reversing all the signs of the edges incident with $x$ always yields an isomorphic double cover. On the other hand, for most graphs there are many signings of a graph which yield non-isomorphic graphs. For example, signing the edges of $\Gamma$ with all $+$ yields two disconnected copies of $\Gamma$, and signing with all $-$ signs yields a bipartite graph, and these are usually not isomorphic (for example, if $\Gamma$ is connected and not bipartite this is clear).

Given a signing $\sigma$ of a graph $\Gamma$ and its corresponding cover $\hat{\Gamma}$ of $\Gamma$, we define $\check{\Gamma}$ to be the double cover of $\Gamma$ obtained by reversing the signs on all the edges of $\Gamma$ (note that $\hat{\Gamma}$ and $\check{\Gamma}$ have the same vertex set $\hat{V}$).  We define $\overline{\Gamma}$ to be the graph with the same vertex set as $\hat{\Gamma}$, where $(x,i) \sim (y,j)$ if and only if $x \neq y$ and $x \nsim y$ in $\Gamma$. Lastly, we define $\Gamma^-$ to be the graph with vertex set $\hat{V}$, where $(x,i) \sim (y,j)$ if and only if $x=y, i \neq j$. Together, the graphs $\hat{\Gamma}, \overline{\Gamma}, \check{\Gamma}, \Gamma^-$ form an edge partition of the complete graph on $\hat{V}$.   

In this paper we will be using the graphs $\hat{\Gamma}, \overline{\Gamma}, \check{\Gamma}, \Gamma^-$ to define an association scheme. The next section outlines a simple condition which we will use to give feasibility criteria for such schemes. 

\subsection{A feasibility condition}


Let $\Gamma$ be a graph with vertices $V$ and edges $E$.  In the following, we will consider subsets of edges of $\Gamma$. For a signed graph $\Gamma,\sigma$, and any set of edges $b$ of $\Gamma$, we define the \textit{weight} of $b$ to be the number of minus signs in $b$. 

The following Lemma is like the pigeonhole principle, in that it is quite simple, but has subtle applications: 

\begin{lemma}\label{pigeon}
Let $\Gamma,\sigma$ be a signed graph, and $X$ a set whose elements are sets of edges of $\Gamma$, where $|X|$ is odd, and each edge of $\Gamma$ is in an even number of elements of $X$. Then some element of $X$ has even weight. 
\end{lemma}

\begin{proof} Since every edge is in an even number of elements of $X$, $\sum\limits_{b \in X} w(b)$ must be even. If every element of $b$ had odd weight, $\sum\limits_{b \in X} w(b)$ would be odd, a contradiction. \end{proof}

The application we explore will be for hypothetical association schemes.  For this application, we will consider certain special sets of edges. Given vertices $x,$ with $x\neq y$, we define the \textit{bundle} $B(x,y)$ to be the set of edges $\{ \{w,z\} : w \in \{x,y\}, z \in V, x \sim z \sim y \}$. Simply put, $B(x,y)$ is the set of edges in the bipartite graph between $x,y$ and their common neighbors. Note $B(x,y) = B(y,x)$. We also define the \textit{augmented bundle} $B^+(x,y)$ int he following manner: for adjacent $x,y$, we have $B^+(x,y) = B(x,y) \cup \{\{x,y\}\}$, and for nonadjacent $x,y$ we have $B^+(x,y)= B(x,y)$. Given a signing $\sigma$ of the graph $\Gamma$ we define the \textit{weight} of the bundle or augmented bundle $b$ to be the number of minus signs on the edges in $b$. We will denote the weight of $b$ by $w(b)$. We refer to a pair $\{x,y\}$ where $x \neq y$ and $x \nsim y$ as a \textit{nonedge} of $\Gamma$. 

The theorem rephrased in terms of these special subsets is below. 

\begin{lemma}\label{bundle}
Let $\Gamma,\sigma$ be a signed graph, and $X$ a set whose elements are bundles or augmented bundles of $\Gamma$, where $|X|$ is odd, and each edge of $\Gamma$ is in an even number of elements of $X$. Then some element of $X$ has even weight. 
\end{lemma}

In the next section we will introduce the needed definitions for association schemes, and we will apply the above theorem to show nonexistence of some hypothetical schemes. 

\subsection{ Association Schemes}

We will be studying association schemes which are, in some sense, generated by a double cover of a graph. 
We begin with the basic definition of association schemes, for more details see \cite{BCN} and \cite{BI}.  
Some sources have a more general definition of association scheme, our definition of association scheme corresponds  to a \textit{symmetric association scheme} in those sources. 

\begin{defn} An association scheme is a set together with a collection of relations $\Rel = \{ R_0, \dots, R_d \}$ on $X$ such that:

\begin{enumerate}
\item $\Rel$ is a partition of $X \times X$
\item $R_0$ is the identity relation on $X$.
\item $R_i = R_i^{\top}$ for $0 \leq i \leq d$
\item there are constants $p_{ij}^k$ such that for any pair $(x,y) \in R_k$ the number of $z \in X$ such that $(x,z) \in R_i$ and $(z,y) \in R_j$ equals $p_{ij}^k$.
\end{enumerate}

\end{defn}

The relation $R_0$ is often referred to as the trivial relation. 
Since the nontrivial relations are symmetric and anti-reflexive, they can be viewed as graphs (see \cite{BCN} for more on graphs). We will henceforth refer to them as graphs instead of relations. When giving the defining graphs of a scheme, we will also omit the trivial relation $R_0$ from the list.

By property $1$ in the above, every pair of vertices is in exactly one of the relations. We write $x \rl{i} y$ when $(x,y) \in R_i$. We now define the adjacency matrices of an association scheme.

\begin{defn} Let $\A = \Adj$ be $(n \times n)$ matrices indexed by $X$ with entries defined by 
$(A_i)_{xy} = 1$ if $x \rl{i} y$ and $(A_i)_{xy} = 0$ otherwise.
\end{defn}

Note that the $A_i$ are all symmetric and linearly independent, and, with the exception of $A_0$, each $A_i$ may be viewed as the adjacency matrix of the corresponding graph.
These matrices satisfy the following:

\begin{enumerate}  

\item $\sum_i A_i = J$
\item $A_0 = I$
\item $A_i = A_i^T$
\item $A_iA_j = \sum_{k} p_{ij}^{k} A_k$
\end{enumerate}

From these properties, it is easily verified that $\BM =  span< A_0, \dots, A_d >$ is a $d+1$ dimensional commutative matrix algebra.  This algebra has the unusual property that it is closed under two distinct multiplications, namely matrix multiplication and the Schur (entrywise) product, and is called the \textit{Bose-Mesner algebra} of the scheme. 

This algebra has another basis of primitive idempotents $E_0, E_1, \dots E_d$, where $E_0 = \frac{1}{|X|}J$. 
There are many uses for the idempotents of a scheme, for more information see \cite{BCN}, chapter 2. 

\subsection{Parameters of association schemes}

The Bose-Mesner algebra $\BM$ has several useful parameters associated with it. In particular, the parameters $p_{ij}^k$ will play an important role. For completeness, we will also include other parameters of the scheme, so we will breifly describe their meaning. 

The left-regular action of the basis $A_0, \dots, A_d$ on itself via matrix multiplication yields a collection of $(d+1) \times (d+1)$ matrices $L_0, \dots, L_d$, defined by $(L_i)_{kj} = p^k_{ij}$, with the property that $L_iL_j = \sum_{k} p_{ij}^{k} L_k$.  These matrices are convenient ways to store the $p_{ij}^k$, since they also share eigenvalues with the corresponding $A_i$. We will refer to the as the intersection matrices of the scheme. 

We define the matrix $P$ to be the $(d+1) \times (d+1)$ matrix whose entries satisfy $A_j E_i = P_{ij} E_i$.  In particular, $P_{ij}$ is precisely the eigenvalue of $A_j$ on the $j$th common eigenspace of $\BM$. We define the matrix $Q$ by $Q=|X| P^{-1}$.  It is well known that the first row of $Q$ gives the dimensions of the common eigenspaces of $\BM$, thus are positive integers which we call the multiplicities of the scheme. 

We will also describe Krein parameters of schemes $q_{ij}^k$.  These are given for completeness, however we will not use them explicitly in the present work, and so omit the definition of them and the dual intersection matrices $L_i^*$, which can be found in chapter 2 of \cite{BCN}.  We note that all the parameters $P_{ij}, Q_{ij}, q_{ij}^k$ can be derived from the $p_{ij}^k$, see \cite{BCN} chapter 2 for details on these calculations. 


The simplest interesting examples of association schemes are given by \textit{strongly regular graphs}. A graph $\Gamma$ is called strongly regular with parameters $(v,k,\lambda,\mu)$ provided that $\Gamma$ is not complete or empty, has $v$ vertices, is $k$-regular, and every adjacent pair of vertices in $\Gamma$ have $\lambda$ common neighbors, and every two distinct non-adjacent vertices have $\mu$ common neighbors. A strongly regular graph $\Gamma$ and its complement form the nontrivial relations of an association scheme with $p_{11}^0=k$, $p_{11}^1=\lambda$ and $p_{11}^2 = \mu$. We call the strongly regular graph $\Gamma$ primitive provided that both $\Gamma$ and its complement are connected graphs. Using the relations of \cite{BCN}, chapter 2 Lemmas 2.1.1 through 2.2.1, all other $p_{ij}^k$ and the matrices $P$ and $Q$ can be computed.  We give $L_1, L_2$ and $P$ below: 

$$L_1 = \begin{pmatrix}
0&k&0\\
1&\lambda&k-\lambda-1 \\
0&\mu&k-\mu \\
\end{pmatrix}
L_2 = \begin{pmatrix}
0&0&v-k-1\\
0&k-\lambda-1&v-2k+\lambda \\
1&k-\mu&v-2k+\mu-2 \\
\end{pmatrix}
P = \begin{pmatrix}
1&k&v-k-1\\
1&r&-r-1 \\
1&s&-s-1 \\
\end{pmatrix}$$

where $r,s = \frac{\lambda-\mu \pm \sqrt{ (\lambda-\mu)^2 + 4(k-\mu)}}{2}$ are the eigenvalues of the SRG, respectively. 

Primitive strongly regular graphs are well-studied, for example see the recent book of Brouwer and Van Maldegham $\cite{SRG}$. We consider primitive strongly regular graphs $\Gamma$ which have double covers $\hat{\Gamma}$, such that $\hat{\Gamma}, \overline{ {\Gamma}}, \check{\Gamma}$ and $\Gamma^-$ form the graphs of a four-class association scheme. We will henceforth call this a 4-class double cover of the graph $\Gamma$. 

Let $\Gamma$ be a primitive $(v,k,\lambda,\mu)$ strongly regular graph with a 4-class double cover. The parameters of the double cover are given below.  We note that the parameters of the 4-class scheme are entirely determined by $v,k,\lambda,\mu$ and the additional parameter $p_{11}^1$. We will label $\lambda_1= p_{11}^1$, and define $\lambda_2 = p_{11}^3$, and note that $\lambda = \lambda_1+\lambda_2$.  The parameters of the scheme and its $P$-matrix are given below.  

\bigskip

$L_1 = \begin{pmatrix} 0&k&0&0&0 \\
1&\lambda_1&k-\lambda-1&\lambda_2&0 \\
0&\frac{\mu}{2}&k-\mu&\frac{\mu}{2}&0 \\
0&\lambda_2&k-\lambda-1&\lambda_1&1 \\
0&0&0&k&0 
\end{pmatrix} \, \, \, \,  L_2 = \begin{pmatrix}
0&0&2(v-k-1)&0&0 \\
0&k-\lambda-1&2(v-2k+\lambda)&k-\lambda-1&0 \\
1&k-\mu&2(v-2k+\mu-2)&k-\mu&1 \\
0&k-\lambda-1&2(v-2k+\lambda)&k-\lambda-1&0 \\
0&0&2(v-k-1)&0&0 \\
\end{pmatrix}$

\bigskip

$L_3 = \begin{pmatrix}
0&0&0&k&0 \\
0&\lambda_2&k-\lambda-1&\lambda_1&1 \\
0&\frac{\mu}{2}&k-\mu&\frac{\mu}{2}&0 \\
1&\lambda_1&k-\lambda-1&\lambda_2&0 \\
0&k&0&0&0 \\
\end{pmatrix} \, \, \, \, L_4 = \begin{pmatrix}
0&0&0&0&1 \\
0&0&0&1&0 \\
0&0&1&0&0 \\
0&1&0&0&0 \\
1&0&0&0&0 \\
\end{pmatrix}$

\bigskip

$P =\begin{pmatrix}
1&k&2(v-k-1)&k&1 \\
1&\frac{( \lambda_1 - \lambda_2 )+\sqrt{ ( \lambda_1 - \lambda_2 )^2 +4k }}{2}&0&\frac{( \lambda_2 - \lambda_1 )-\sqrt{ ( \lambda_1 - \lambda_2 )^2 +4k }}{2}&-1 \\
1&r&-2(r+1)&r&1 \\
1&\frac{( \lambda_1 - \lambda_2 )-\sqrt{ ( \lambda_1 - \lambda_2 )^2 +4k }}{2}&0&\frac{( \lambda_2 - \lambda_1 )+\sqrt{ ( \lambda_1 - \lambda_2 )^2 +4k }}{2}&-1 \\
1&s&-2(s+1)&s&1 \\
\end{pmatrix}$

We note that all the graphs $\overline{ {\Gamma}}, \check{\Gamma}$ and $\Gamma^-$ are determined by the double cover $\hat{\Gamma}$. Thus checking whether these graphs form a 4-class scheme can be reduced to conditions on $\hat{\Gamma}$. We summarize these conditions below in the following theorem: 

\begin{thm}
    Let $\Gamma$ be a $(v,k,\lambda,\mu)$ SRG, and let $\hat{\Gamma}$ be a double cover of $\Gamma$. Then the graphs $\hat{\Gamma}, \overline{ {\Gamma}}, \check{\Gamma}$ and $\Gamma^-$ form the graphs of a 4-class association scheme if and only if the parameters $p_{11}^1$, $p_{11}^2$ are well-defined. That is, if and only if every adjacent pair of vertices in $\hat{\Gamma}$ has the same number of common neighbors, and every pair of non-adjacent vertices in $\hat{ {\Gamma}}$ whose images are distinct and non-adjacent in $\Gamma$ has the same number of common neighbors. 
\end{thm}

\begin{proof}
    Since a strongly regular graph together with its complement gives an association scheme in its own right, it will be convenient to refer to the parameters of the scheme of the SRG by $r_{ij}^k$. Thus $r_{11}^1 = \lambda, r_{11}^2 = \mu$, while saving the notation $p_{ij}^k$ for the parameters of the putative four class scheme. 
    
    For a vertex $x$ of $\hat{\Gamma}$, let $x'$ denote its antipode, and note $x''=x$.  
    Let $\hat{\Gamma}, \overline{ {\Gamma}}, \check{\Gamma}, \Gamma^-$ be the relations $R_1,R_2,R_3,R_4$, and for $x,y \in V$ we say that $x \rl{i} y$ provided that $x \sim y$ in relation $i$.  Let $v_i$ denote the valency of relation $i$, and with $v_0=1$. Note that the parameters $p_{ij}^k$ where at least one of $i,j,k$ are equal to zero are all easily computed and well-defined from the fact that $\hat{\Gamma}, \overline{ {\Gamma}}, \check{\Gamma}, \Gamma^-$ are regular graphs that edge partition the complete graph.

    We now show all parameters $p_{ij}^k$ are well-defined. By symmetry we have $p_{ji}^k=p_{ij}^k$ when $p_{ij}^k$ is well-defined. Also, if for fixed $i,k$ we have $p_{ij}^k$ is well-defined for all except possible one value of $j$, then the $p_{ij}^k$ for the last value of $j$ must also be well-defined, and $\sum\limits_j p_{ij}^k = v_i$. 
    
    Let $g = (04)(13)$ be a permutation of $Sym(\{0,1,2,3,4\})$. We have $x \rl{i} y$ if and only if $x \rl{i^g} y'$, which we will use repeatedly. Thus if $p_{ij}^k$ is well-defined, we have $p_{i^gj^g}^k = p_{ij}^k$ and $p_{ij^g}^{k^g} = p_{i^gj}^{k^g} = p_{ij}^k$ are also well-defined. In particular all parameters $p_{ij}^k$ where at least one of $i,j,k$ is equal to 4 are well defined, since $p_{4j}^k = p_{0j^g}^k$, etc. We also have $p_{11}^1 = p_{33}^1=p_{13}^3=p_{31}^3$ are well-defined. We have $p_{22}^2 = 2r_{22}^2$ and $p_{22}^1=p_{22}^3=2r_{22}^1$ from the definition of $\overline{\Gamma}$.  

    Since $p_{11}^2$ is well-defined, we also have $p_{11^g}^{2^g} = p_{13}^2=p_{31}^2 = p_{33}^2$. This forces $p_{12}^2 = k-p_{11}^2-p_{13}^2-p_{14}^2$ to be well-defined, and similarly $p_{32}^2, p_{21}^2, p_{23}^2$.

    Now we consider the parameter $p_{12}^1$. Let $x=(w,a), y=(z,b)$ be vertices of $\hat{\Gamma}$ with $x \rl{1} y$, where $w,z$ are vertices of $\Gamma$ and $a,b \in \{1,2\}$.  Since $x \rl{1} y$, we must have $w \sim z$ in $\Gamma$. Let $u$ be a vertex with $u \neq z$ and $w \sim u \not\sim z$, and note there are $r_{12}^1$ choices for $u$.  We have $(u,a) \rl{1} x$ for exactly one choice of $a$, and $(u,a) \rl{2} y$ for both choices of $a$. Thus $p_{12}^1 = r_{12}^1$, and so we find $p_{21}^1, p_{32}^3, p_{23}^3$. This now forces $p_{21}^3=p_{12}^3$ and $p_{23}^1=p_{32}^1$, since all other parameters of the form $p_{2j}^3$ and $p_{2j}^1$ are well-defined. 

    Similarly, this forces $p_{13}^1$ to be well-defined since all other $p_{1j}^1$ are, so $p_{31}^1$ is as well. Lastly, applying $g$, we also obtain $p_{33}^3, p_{11}^3$. 
    
\end{proof}

In the particular case where $\lambda_2=0$ the graph $\hat{\Gamma}$ is a \textit{distance-regular graph}. In particular, $\hat{\Gamma}$ is a diameter 4 antipodal distance regular graph with intersection array $\{k,k-\lambda-1,\frac{\mu}{2},1;1,\frac{\mu}{2},k-\lambda-1,k  \}$.  For more on distance-regular graphs, see \cite{BCN}.  In particular, on pages 421-425 there is a table of feasible intersection arrays of hypothetical diameter 4 antipodal distance-regular graphs, including what was known at the time about existence and uniqueness of each array. See the dynamic survey \cite{DRGsurvey} pages 105-106 for the most current updates to this table.

\subsection{Example on Ordered Pairs}

Let $n \geq 5$ be a positive integer $\Gamma$ be the triangular graph $T(n)$.  This graph has as vertices the set of all two element subsets of $\{1..n\}$, with two vertices adjacent when they intersect in 1 element. This is a well-known family of primitive 
$( \binom{n}{2},  2(n-2), n-2,4 )$ SRGs, see \cite{SRG}, for example. These graphs have a natural 4-class double cover, which we describe below. 

Let $\hat{\Gamma}$ be the graph with vertex set $V = \{ (x,y) : x,y \in \{1,\dots,n\}, x \neq y \}$, with two distinct ordered pairs $(x,y), (w,z)$ adjacent when $x=w$ or $y=z$, but not both. 
This is easily seen to be a double cover of $T(n)$ with fibers of the form $\{ (x,y), (y,x) \}$. This generates a 4-class scheme, with parameters given below: 

\bigskip

$L_1 = \begin{pmatrix} 0&2n-4&0&0&0 \\
1&n-3&n-3&1&0 \\
0&2&2n-8&2&0 \\
0&1&n-3&n-3&1 \\
0&0&0&k&0 
\end{pmatrix}  \, \, \, \, L_2 = \begin{pmatrix}0&0&n^2-5n+6&0&0 \\
0&n-3&n^2-7n+12&n-3&0 \\
1&2n-8&n^2-9n+20&2n-8&1 \\
0&n-3&n^2-7n+12&n-3&0 \\
0&0&n^2-5n+6&0&0 \\
\end{pmatrix}$


\bigskip

$L_3 = \begin{pmatrix}
0&0&0&2n-4&0 \\
0&1&n-3&n-3&1 \\
0&2&2n-8&2&0 \\
1&n-3&n-3&1&0 \\
0&2n-4&0&0&0 \\
\end{pmatrix} \,  \,  \,  \, L_4 = \begin{pmatrix}
0&0&0&0&1 \\
0&0&0&1&0 \\
0&0&1&0&0 \\
0&1&0&0&0 \\
1&0&0&0&0 \\
\end{pmatrix}$

\bigskip

$P =\begin{pmatrix}
1&2n-4&n^2-5n+6&2n-4&1 \\
1&n-2&0&-n+2&-1 \\
1&n-4&-2n+6&n-4&1 \\
1&-2&0&2&-1 \\
1&-2&2&-2&1 \\
\end{pmatrix}$

\bigskip

This example appears a number of times in the table of Section \ref{ne}. 

\section{A linear system for 4-class double covers}

The first application we give of Lemma \ref{bundle} is in the form of a linear system of equations of $\mathbb{F}_{2}$. 

Let $\Gamma$ be a $(v,k,\lambda,\mu)$-strongly regular graph with edge set $E$.  For each edge $e \in E$ we assign a corresponding binary variable $x_e$. Given a set of edges $X \subset E$ we define $\varsigma( X) = \sum\limits_{e \in X} x_e $.  

\begin{thm}\label{linsys}

Let $W$ be the set of nonedges of the $(v,k,\lambda,\mu)$-strongly regular graph $\Gamma$ and $B(W)$ the associated bundles. Let $E$ be the set of edges of $\Gamma$ and $B(E),B^+(E)$ the associated set of associated bundles/augmented bundles, respectively. Let $\lambda_1,\lambda_2$ be nonegative integers with $\lambda_1+\lambda_2=\lambda$. 
Let $S_1$ be the set of equations of the form $$\varsigma( b) =  \frac{\mu}{2} \text{ (mod 2)}$$ for each $b \in B(W)$ and where the modulus is used as an operator in this context (as opposed to a relation). 
Let $S_2$ be the set of equations of the form $$\varsigma( b) = \lambda_2 \text{ (mod 2)}$$ for each $b \in B(E)$.

Lastly, let $S_3$ be the set of equations of the form $$\varsigma( b) = \lambda_2 \text{ (mod 2)}$$ for each $b \in B^+(E)$.

If $\Gamma$ has a 4-class double cover, then the system $S_1$ has a solution. 

In particular, if $\Gamma$ has a 4-class double cover with $p_{11}^1 = \lambda_1$ and $\lambda_2=\lambda-\lambda_1$, then

    \begin{itemize}
        \item If $\lambda$ is even, the system $S_1 \cup S_2$ has a solution.
        \item If $\lambda$ is odd the system $S_1 \cup S_3$ has a solution.
    \end{itemize}

\end{thm}

\begin{proof}

Consider a signing $\sigma$ of the edges of $\Gamma$ that produces the desired double cover $\hat{\Gamma}$. 
For each edge $e$, let $x_e=0$ if $\sigma(e)=+$, and let $x_e=1$ if $\sigma(e)=-$.
Let $\{u,v\}$ be a nonedge of $\Gamma$, and consider the vertices $(u,1),(v,1)$ of $\Gamma$.  Noting that $p_{11}^2 = \frac{\mu}{2}$, we must have that exactly $\frac{\mu}{2}$ of the two paths from $u$ to $v$ have both edges the same sign, and exactly $\frac{\mu}{2}$ of the two paths from $u$ to $v$ have edges with opposite signs. Therefore $\varsigma( B(x,y) ) =  \frac{\mu}{2} \text{ (mod 2)}$.  

Let $e=\{ u,v\}$ be an edge of $\Gamma$. If $x_e=0$, then exactly $\lambda_2$ two-paths from $u$ to $v$ have edges with opposite signs, thus $\varsigma( B(u,v) ) =  \lambda_2 \text{ (mod 2)}$. If $x_e=1$, then exactly $\lambda_1$ two-paths from $u$ to $v$ have edges with opposite signs, thus $\varsigma( B(u,v) ) =  \lambda_1 \text{ (mod 2)}$. 

For even $\lambda$, we must have $\lambda_1 \equiv \lambda_2 \text{ (mod 2)}$, thus we obtain $\varsigma( B(u,v) ) =  \lambda_2 \text{ (mod 2)}$ regardless of the sign of $e$.

For odd $\lambda$, we have $\lambda_1+1 \equiv \lambda_2 \text{ (mod 2)}$, and $\varsigma( B^+(u,v) ) = x_e+\varsigma( B(u,v) )$.  Thus regardless of the sign of $e$, we have $\varsigma( b) = \lambda_2 \text{ (mod 2)}$. 

\end{proof}

Given the strongly regular graph and parameters of a hypothetical 4-class double cover, the solution space of this linear system can be a useful tool for searching for the scheme. In the case of girth 5 distance-regular graphs, the system $S_1$ can be used to verify the non-existence claim for a $DRG-\{10,9,1,1;1,1,9,10\}$ in \cite{BCN} p. 372 in the first paragraph of the proof of Proposition 11.4.5 (in this case, the 4-cycles are equal to what we call bundles).  Similar linear systems were also used by Moorhouse in \cite{M1} to create new projective planes from semibiplanes. 

The original motivation behind this line of research was to use this system to conduct searches for open 4-class double covers.  However, in the case where $\mu \equiv 2 \text{ (mod 4) }$ or $\lambda_2 \equiv 1 \text{ (mod 2) }$ it is possible the system could be inconsistent. This condition turns out to be fruitful in and of itself, and thus we focus our attention on conditions that we can use to rule out 4-class double covers, which we give in the next section. Often these are more easily phrased without a direct appeal to the linear system, and instead in terms of Lemma \ref{bundle}, but it is implicit, as every nonexistence result in the next section implies the respective linear system of Theorem \ref{linsys} is inconsistent.

\section{Necessary conditions for 4-class double covers}

We now consider $4$-class double covers of a strongly regular graph ${\Gamma}$, and corollaries of Lemma \ref{bundle}.  We have some separate cases to consider. 

\begin{cor}\label{cor1}
    Let $\Gamma$ be a strongly regular graph, where $\mu \equiv 2 \text{ (mod 4)}$. Let $W$ be a set of of nonedges of $\Gamma$, and let $B(W)$ be the set of bundles corresponding to the nonedges in $W$. If $|B(W)|$ is odd, and every edge of $\Gamma$ is in an even number of bundles in $B(W)$, then $\Gamma$ has no 4-class double cover. 
\end{cor}

\begin{proof} Suppose a 4-class double cover of such a graph exists. We use 
Lemma \ref{bundle} with $W$ equal to the set of all nonedges of $\Gamma$. We have $|W| $ is odd, thus $|B(W)|$ is odd. 
Every edge is in an even number of bundles in $B(W)$ by our hypothesis. Let $\{x,y\} \in W$ be a nonedge, and $b \in B(W)$ the corresponding bundle. We have that $b$ contains $2\mu$ edges, which form $\mu$ edge-disjoint two-paths from $x$ to $y$. In the corresponding 4-class double cover, the parameter $p_{13}^1 = \frac{\mu}{2}$ is odd. Thus $\frac{\mu}{2}$ of these two-paths must have edges with the same sign, and $\frac{\mu}{2}$ of these two paths must have edges with opposite signs. 
Thus the number of edges in $b$ with negative signs is congruent to $\frac{\mu}{2} \text{ (mod 2)}$, which is odd.  This contradicts Lemma \ref{bundle}.
\end{proof}

\bigskip

We note that in most cases $|B(W)| = |W|$, that is, there is a correspondence between bundles/augmented bundles and the nonedge/edge that generates them.  The only exceptions 
occur when: 

\begin{enumerate} 

\item $\mu = 2$ and $\Gamma$ has a chordless 4-cycle $\{x_1,x_2,x_3,x_4\}$, in which case $B(\{x_1,x_3\}) = B(\{x_2,x_4\}) $

\item $\lambda = 2$ and $\Gamma$ has a $K_4$ with vertices $\{x_1,x_2,x_3,x_4\}$, in which case $B(\{x_1,x_2\}) = B(\{x_3,x_4\})$,  $B(\{x_1,x_3\}) = B(\{x_2,x_4\})$, $B(\{x_1,x_4\}) = B(\{x_2,x_3\})$

\item $\lambda = 1$ in which case any triangle $\{x_1,x_2,x_3\}$ satisfies $B^+(\{x_1,x_2\}) = B^+(\{x_1,x_3\}) = B^+(\{x_2,x_3\})$ 

\item $\lambda = 0$ where the bundles of adjacent vertices are empty

\item when $\lambda = \mu = 2$ and $\{x_1,x_2,x_3,x_4\}$ induces a kite, that is a $K_4$ minus the edge $\{x_1, x_3\}$, in which case the bundle of the nonedge $\{x_1, x_3\}$ is the same as the bundle of the edge $\{x_2, x_4\}$.

\end{enumerate}

Thus we must give special care to the cases where $\mu=2$ or $\lambda = 0,1,2$.
As an example, we examine 4-class double covers of the grid graphs (see \cite{SRG}) which are $SRG-(n^2, 2(n-1), n-2, 2)$. 

\bigskip  

\begin{cor}\label{grid}

    Let $n \geq 3$ and $\Gamma$ be the grid graph of an $n \times n$ grid. Then $\Gamma$ has no $4$-class double cover. 

\end{cor}

\begin{proof}  Let $W$ equal to the set of all nonedges of the subgraph of $\Gamma$ induced by a $3 \times 3$ subgrid. We have $|W| = 18$.  Since every 4-cycle of $\Gamma$ is chordless and $\mu=2$, $|B(W)| = 9$ which is odd.  The number of bundles of $B(W)$ containing a fixed edge is precisely $2$ or $0$, depending on whether the edge is in the fixed subgrid or not. The result now follows from Corollary \ref{cor1}.
\end{proof}

\bigskip

When $\mu$ is singly even, the entire set of nonedges sometimes yields a contradiction. This is summarized in the next Corollary. 

\begin{cor}\label{corDRG}

    Let $\Gamma$ be a $(v,k,\lambda,\mu)$ strongly regular graph, where $v(v-k-1) \equiv \mu \equiv 2 \text{ (mod 4)}$ and $\mu \neq 2$.  Then $\Gamma$ has no $4$-class double cover. 

\end{cor}

\begin{proof}  Let $W$ equal the set of all nonedges of $\Gamma$. We have $|W| = \frac{v(v-k-1)}{2}$, which is odd, thus $|W| = |B(W)|$ is odd.  The number of bundles of $B(W)$ containing a fixed edge is precisely $2(k-\lambda-1)$, which is always even. The result now follows from Corollary \ref{cor1}.
\end{proof}

Using Corollary \ref{corDRG}, we can rule out a certain parameter family of distance-regular graphs: 

\begin{cor}

    There are no diameter 4 antipodal distance-regular graphs with intersection array 
    $\{k,b_1,b_2,1;1,b_2,b_1,k \}$ where $b_2 \neq 1$ and $\frac{ k b_1 }{4}(1+k+\frac{ k b_1 }{2b_2})$ is odd. 

\end{cor}

We give an example of a two parameter family of arrays ruled out by this:

\begin{cor}

For positive integers $\lambda$ and $w$ with $\lambda \equiv 3 \text{ (mod 4) }$, and $w \equiv 3,5 \text{ (mod 8) }$, there is no distance-regular graph with an  intersection array of the form:  $$\left\{ \frac{w^2(\lambda+2)^2-\lambda^2}{4}, \frac{(w^2-1)(\lambda+2)^2}{4}, \lambda+2,1; 
1, \lambda+2, \frac{(w^2-1)(\lambda+2)^2}{4}, \frac{w^2(\lambda+2)^2-\lambda^2}{4} \right\}$$

\end{cor}

Specializing $w=3$ and $\lambda = 7$ , we obtain nonexistence for the array $\{ 170,162,9,1;,1,9,162,170 \}$, which was previously unresolved (see \cite{DRGsurvey}). 

For the special case $w=\lambda \equiv 3 \text{ (mod 8)}$ we have the nonexistence of the intersection arrays $$\left\{ \frac{\lambda^2((\lambda+2)^2-1)}{4}, \frac{\lambda^2((\lambda+2)^2-1)}{4}-\lambda-1, \lambda+2,1; 
1, \lambda+2, \frac{\lambda^2((\lambda+2)^2-1)}{4}-\lambda-1, \frac{\lambda^2((\lambda+2)^2-1)}{4} \right\}.$$ In particular, this resolves the cases $\lambda=3$ corresponding to the intersection array
$\{ 54,50,5,1;,1,5,50,54 \}$ (see \cite{DRGsurvey}).

\begin{cor}\label{lameven}

    Let $\Gamma$ be $(v,k,\lambda,\mu)$ strongly regular graph, where $\lambda$ is even. Let $W$ be a set of of edges of $\Gamma$, and let $B(W)$ be the set of bundles corresponding to the edges in $W$. If $|B(W)|$ is odd, and every edge of $\Gamma$ is in an even number of bundles in $B(W)$, then $\Gamma$ has no 4-class double cover with $\lambda_1 = p_{11}^1$ odd. 

\end{cor}

\begin{proof}  Use Lemma \ref{bundle} with $B(W)$ as described in the theorem. 
\end{proof}





\begin{cor}\label{lamnbhd}

    Let $\Gamma$ be $(v,k,\lambda,\mu)$ strongly regular graph, where $k$ is odd and $\lambda \equiv 2 \text{ (mod 4) }$. Then $\Gamma$ has no 4-class double cover with $p_{11}^1$ odd. 

\end{cor}

\begin{proof}  Let $x$ be a vertex of $\Gamma$, and let $W$ be equal to the set of all edges of $\Gamma$ containing $x$. The number of bundles is $k$ which is odd, and each edge $e$ is in $\lambda,2$ or $0$ bundles depending on whether $x$ is in $e$, or $x$ is not in $e$ and both vertices of $e$ are adjacent to $x$, or otherwise. The result follows from Corollary \ref{lameven}.
\end{proof}


\bigskip

The following theorem concerns the strongly regular graph families $V^{\epsilon}(2m,q)$, see \cite{SRG} for more details.
Given a natural number $m$, a prime power $q$ and an $\epsilon \in \{-1,1\}$, we define $V^{\epsilon}(2m,q)$ as follows.  Let $V = \mathbb{F}_q^{2m}$ and $Q(x,y)$ a nondegenerate quadratic form on $V$ of type $\epsilon$. For $x,y \in V$, we define adjacency by $x \sim y$ if and only if $x \neq y$ and $Q(x-y)=0$. The graph $V = \mathbb{F}_q^{2m}$ is known to be strongly regular with parameters $(q^{2m}, (q^m-\epsilon)(q^{m-1}+\epsilon), q(q^{m-1}-\epsilon)(q^{m-2}+\epsilon)+q-2, q^{m-1}(q^{m-1}+\epsilon))$.

We obtain the following theorem: 

\begin{thm}\label{Patv2}

    The strongly regular graphs $VO^+(2m,q)$ with with $m$ odd and $q \equiv 3 \text{ (mod 4)}$ have no 4-class double covers.  Similarly, the strongly regular graphs $VO^-(2m,q)$ with with $m$ even and $q \equiv 3 \text{ (mod 4)}$ have no 4-class double covers. 
    In particular, there is no DRG-$\{20,18,3,1;1,3,18,20\}$.
    
\end{thm}

\begin{proof}
We utilize Corollary \ref{cor1}.  Let $q$ be a prime power congruent to 3 modulo 4, and $m$ odd or even depending on whether $\epsilon =1,-1$, respectively. Let $V = \mathbb{F}_q^{2m}$, and let $Q(x)$ be a nondegenerate quadratic form of the appropriate type on $V$, and let $\Gamma = VO^\epsilon(2m,q)$. For $\Gamma$, we have $\mu = q^{m-1}(q^{m-1}+\epsilon) \equiv 2 \text{ (mod 4)}$ for both cases. Let $W$ be the set of nonedges $\{x,y \}$ such that $Q(x-y)$ is a nonzero square. It is known that for nonzero $b \in \mathbb{F}_q$ the number of solutions to $Q(x) = b$ is independent of the choice of $b$ (see Theorem 6.26 in \cite{FF}, page 282). In particular, $|W|$ is equal to exactly half the number of nonedges of the graph. 

Thus $|W| = \frac{1}{4}q^{2m} (q^{2m} - (q^m-\epsilon)(q^{m-1}+\epsilon) - 1) = \frac{1}{4} q^{2m} q^{m-1}(q-1)(q^m-\epsilon)$.  The quantity $q^{2m} q^{m-1}(q-1)(q^m-\epsilon) \equiv 4 \text{ (mod 8)}$ for the cases satisfying the hypothesis of the theorem, thus $|W|$ is odd. 




Let $\{ u, v \}$ be an edge of $\Gamma$.  We claim that each edge of $\Gamma$ is contained in an even number of bundles of $B(W)$.  In order to show this, we use affine transformations of $V$ which are also graph automorphisms of $\Gamma$.  For $a \in V$, let $\varphi: V \to V$ be given by $\varphi_a(x) = x + a$.  It is clear that $\varphi_a$ preserves $Q$ and therefore is an automorphism of $\Gamma$ and leaves $W$ invariant.  For $\alpha \in \mathbb{F}_q^*$, let $\psi_\alpha: V \to V$ be given by $\psi_\alpha(x) = \alpha x$.  Here it is also clear that $\psi_\alpha$ is an automorphism of $\Gamma$, and since $Q(\psi_\alpha(x)) = Q(\alpha x) = \alpha^2 x$, we again have that $B(W)$ is invariant.  We therefore have that the number of bundles of $B(W)$ containing $\{u, v\}$ is the same as the number of bundles of $B(W)$ containing $\varphi_{-u} (\{ u, v \}) = (0, x)$, where $x = v - u$.  

We may simplify the count further.  Let $S_1 = \{ B(y, z) \in B(W) : \{0, x\} \in B(y, z), 0 \in \{ y, z\} \}$, and let $S_2 = \{ B(y, z) \in B(W) : \{0, x\} \in B(y, z), x \in \{ y, z\} \}$.  Note that $S_1$ and $S_2$ partition the set of bundles of $B(W)$ that contain $\{x, y\}$. We note that $\psi_{-1}(\varphi_{-x}(S_1)) = S_2$, and so $|S_1| = |S_2|$.  Therefore, the desired count is equal to $2|S_1|$, which is even.  Therefore, Corollary \ref{cor1} applies.

Lastly, for the DRG-$\{20,18,3,1;1,3,18,20\}$, this would be a 4-class double cover of a SRG-$(81,20,1,6)$.  It is known that $VO^-(4,3)$ is the unique strongly regular graph with these parameters (see \cite{SRG}). Thus no such distance-regular graph exists. 
\end{proof}

With more effort, the value of $|S_1|$ can be calculated explicitly.  For example, when $m = 2$, in order to count $|S_1|$, we fix $x = (x_1, x_2, xi_3, x_4) \in V$ assume $\{0, x\} \in B(0, y)$, with and $B(0, y) \in S$.  We wish to count the number of $y = (y_1, y_2, y_3, y_4)$ satisfying this condition.  We do this in two cases, and use the quadratic form $Q(x) = x_1^2 + x_2^2 + x_3 x_4$ for convenience.  

We first assume $(x_1, x_2) = (0, 0)$.  Since $Q(x) = 0$, this implies $x_3 x_4 = 0$.  Without loss of generality, we may assume $x_3 = 0$ and therefore $x_4 \neq 0$.  We immediately have $y_1^2 + y_2^2 + y_3 y_4 = s$ and $(x_1 - y_1)^2 + (x_2 - y_2)^2 + (x_3 - y_3)(x_4 - y_4) = 0$.  We simplify this to obtain

\begin{align*}
    x_1^2 - 2x_1 y_1 + y_1^2 + x_2^2 - 2x_2 y_2 + y_2^2 + x_3 x_4 - x_3 y_4 - x_4 y_3 + y_3 y_4 &= 0 \\
    s - 2x_1 y_1 - 2x_2 y_2 - x_3 y_4 - x_4 y_3 &= 0 \\
    2x_1 y_1 + 2 x_2 y_2 + x_3 y_4 + x_4 y_3 &= s \\
    x_4 y_3 &= s.
\end{align*}
The last line follows from the fact that $x_1 = x_2 = x_3 = 0$.  We also have that $y_1^2 + y_2^2 + y_3 y_4 = s$.  Since $y_3 = sx_4^{-1}$, we have $y_1^2 + y_2^2 + sx_4^{-1} y_4 = s$.  By rearranging terms we obtain the equation

\begin{equation*}
    y_1^2 + y_2^2 = s(1 - x_4^{-1} y_4)
\end{equation*}

We wish to count the number of $y$ and $s$ that satisfy this equation.  There are two cases to consider: either $y_4 \neq x_4$ or $y_4 = x_4$.  In the former case there are $q - 1$ ways of choosing $y_4$, and there are $\frac{q-1}{2}$ ways of choosing $s$.  In this way the right-hand side of the above equation is nonzero.  Because of this, there are $q + 1$ ways of choosing $y_1$ and $y_2$.  This means there are $(q + 1)(q - 1)(\frac{q-1}{2})$ total choices for $y$.  

In the case where $y_4 = x_4$ we have $y_1^2 + y_2^2 = 0$.  This forces $y_1 = y_2 = 0$, and we have $\frac{q-1}{2}$ choices for $y_3$.  This means in total there are $(q+1)(q-1)(\frac{q-1}{2}) + \frac{q-1}{2} = q^2 (\frac{q-1}{2})$ choices for $y$.  We conclude that there are $q^2 (\frac{q-1}{2})$ bundles of $S_1$.  

We now consider the case in which $(x_1, x_2) \neq (0, 0)$.  We first chose $y_1, y_2, $ and $s$ in $q^2(\frac{q-1}{2})$ ways.  We immediately have $y_1^2 + y_2^2 + y_3 y_4 = s$ and $(x_1 - y_1)^2 + (x_2 - y_2)^2 + (x_3 - y_3)(x_4 - y_4) = 0$.  We simplify this to obtain

\begin{align*}
    x_1^2 - 2x_1 y_1 + y_1^2 + x_2^2 - 2x_2 y_2 + y_2^2 + x_3 x_4 - x_3 y_4 - x_4 y_3 + y_3 y_4 &= 0 \\
    - 2x_1 y_1 - 2x_2 y_2 - x_3 y_4 - x_4 y_3 + s &=0 \\
    2x_1 y_1 + 2x_2 y_2 + x_3 y_4 + x_4 y_3 = s.
\end{align*}

Using these equations, we now solve for possible values for $y_3$ and $y_4$.  To do this, we let $\alpha = y_3 y_4 = s - y_1^2 - y_2^2$, and $\beta = x_3 y_4 + x_4 y_3 = s - 2x_1 y_1 - 2 x_2 y_2$.  Note that if $y_3 = 0$,then we can solve for a unique solution to $y_4$.  Otherwise, we have $y_4 = \alpha y_3^{-1}$.  By substitution we obtain

\begin{align*}
    \alpha x_3 y_3^{-1} + x_4 y_3 &= \beta \\
    \alpha x_3 + x_4 y_3^2 &= \beta y_3 \\
    x_4 y_3^2 - \beta y_3 + \alpha x_3 &= 0 \\
    y_3 = \frac{\beta \pm \sqrt{\beta^2 - 4\alpha x_3 x_4}}{2x_4}.
\end{align*}

These values of $z_3$ exist either in $\mathbb{F}_q$ or in the quadratic extension $\mathbb{F}_{q^2}$.  This depends on whether $\beta^2 - 4 \alpha x_3 x_4$ is $0$, a nonzero square or a nonsquare in $\mathbb{F}_q$.  We now investigate the quantity $\beta^2 - 4 \alpha x_3 x_4$.  

\begin{align*}
    \beta^2 - 4\alpha x_3 x_4 &= (s - 2(x_1 y_1 + x_2 y_2))^2 - 4(s - (y_1^2 + y_2^2))x_3 x_4 \\
    &= s^2 - 4s(x_1 y_1 + x_2 y_2) + 4(x_1 y_1 + x_2 y_2)^2 - 4sx_3 x_4 + 4x_3 x_4(y_1^2 + y_2^2) \\
    &= s^2 - 4s(x_1 y_1 + x_2 y_2) + 4(x_1^2 y_1^2 + 2x_1 x_2 y_1 y_2 + x_2^2 y_2^2) - 4sx_3 x_4 + 4x_3 x_4(y_1^2 + y_2^2) \\
    &= s^2 - 4s(x_1 y_1 + x_2 y_2) + 4((x_1^2 + x_3 x_4)y_1^2 + 2x_1 x_2 y_1 y_2 + (x_2^2 + x_3 x_4)y_2^2) - 4sx_3 x_4 \\
    &= s^2 - 4s(x_1 y_1 + x_2 y_2) + 4(-x_2^2 y_1^2 + 2x_1 x_2 y_1 y_2 -x_1^2 y_2^2) - 4sx_3 x_4 \\
    &= s^2 - 4s(x_1 y_1 + x_2 y_2) - 4(x_2^2 y_1^2 - 2x_1 x_2 y_1 y_2 + x_1^2 y_2^2) - 4sx_3 x_4 \\
    &= s^2 - 4s(x_1 y_1 + x_2 y_2) - 4(x_2 y_1 - x_1 y_2)^2 - 4sx_3 x_4 \\
    &= s^2 - 4s(x_1 y_1 + x_2 y_2) - 4(x_2 y_1 - x_1 y_2)^2 + 4s(x_1^2 + x_2^2) \\
    &= s^2 + 4s(x_1^2 - x_1 y_1 + x_2^2 - x_2 y_2) - 4(x_2 y_1 - x_1 y_2)^2 \\
    &= s^2 + 4s(x_1(x_1- y_1) + x_2(x_2 - y_2)) - 4(x_2 y_1 - x_1 y_2)^2
\end{align*}

We now change parameters by letting $y_1 = x_1 - a$ and $y_2 = x_2 - b$.  We now obtain

\begin{align*}
    \beta^2 - 4\alpha x_3 x_4 &= s^2 + 4s(ax_1 + bx_2) - 4(x_2(x_1 - a) - x_1(x_2 - b))^2 \\
    &= s^2 + 4s(ax_1 + bx_2) - 4(bx_1 - ax_2)^2.
\end{align*}

We now change parameters again by letting $c = ax_1 + bx_2$ and $d = bx_1 - ax_2$.  Making these substitutions we obtain 

\begin{equation*}
    \beta^2 - 4\alpha x_3 x_4 = s^2 + 4sc - 4d^2.
\end{equation*}

Note that as $c$ and $d$ range over $\mathbb{F}_q$, so too do $y_1$ and $y_2$ as $a$ and $b$ are both nonzero and $q \equiv 3 \mod 4$.  Since $s$ is nonzero and fixed, if we fix $d$, then the above quantity on the righthand side attains all values in $\mathbb{F}_q$ exactly once as $c$ ranges over $\mathbb{F}_q$.  We conclude that in this case, $|S_1| = q^2 (\frac{q-1}{2})$.


\bigskip 

This theorem  also shows that a distance-regular graph with intersection array $\{300, 294, 21, 1, 1, 21, 294, 300\}$ cannot be a double cover of $VO^-(4,7)$. This does not disprove the existence of such a graph, since there may be other strongly regular graphs with the parameters $(2401,300,5,42)$.

\bigskip

In some cases the correct set of bundles may be harder to describe, but the direct checking of the linear system in Theorem \ref{linsys} is more convenient. 
This is used in the following Corollary to rule out a hypothetical antipodal distance-regular graph whose existence was previously unknown (see \cite{BCN}). 

\begin{cor}\label{HS}

    There is no DRG-$\{22,21,3,1;1,3,21,22\}$.
    
\end{cor}

\begin{proof}
    The scheme of a DRG-$\{22,21,3,1;1,3,21,22\}$ would be a 4-class double cover of the Higman-Sims graph, the unique $(100,22,0,6)$ strongly regular graph. By computation using Magma the linear system described in Theorem \ref{linsys} is inconsistent. 
    



\end{proof}

\bigskip
In the next two sections, we will give information on 4-class double covers up to 100 vertices. We will divide 4-class covers into two types, depending on whether $\lambda_1=\lambda_2$ or not. We will also compute the Krein parameters for each case. 

\section{ The case $\lambda_1=\lambda_2$ }

In this section, we assume we have a 4-class double cover of a primitive $(v,k,\lambda,\mu)$ SRG, where $\lambda_1 = \lambda_2$.  Thus $\lambda$ is even, and we find the multiplicities force $v$ to be even as well. The parameters are given in terms of $(v,k,\lambda,\mu)$, and the nontrivial eigenvalues of the strongly regular graph $r,s$ of multiplicities $f,g$, respectively. 

\bigskip

$L_1 = \begin{pmatrix}
0&k&0&0&0 \\
1&\frac{\lambda}{2}&k-\lambda-1&\frac{\lambda}{2}&0 \\
0&\frac{\mu}{2}&k-\mu&\frac{\mu}{2}&0 \\
0&\frac{\lambda}{2}&k-\lambda-1&\frac{\lambda}{2}&1 \\
0&0&0&k&0 \\
\end{pmatrix} \quad L_2 = \begin{pmatrix}
0&0&2(v-k-1)&0&0 \\
0&k-\lambda-1&2(v-2k+\lambda)&k-\lambda-1&0 \\
1&k-\mu&2(v-2k+\mu-2)&k-\mu&1 \\
0&k-\lambda-1&2(v-2k+\lambda)&k-\lambda-1&0 \\
0&0&2(v-k-1)&0&0 \\
\end{pmatrix}$

\bigskip

$L_3 = \begin{pmatrix}
0&0&0&k&0 \\
0&\frac{\lambda}{2}&k-\lambda-1&\frac{\lambda}{2}&1 \\
0&\frac{\mu}{2}&k-\mu&\frac{\mu}{2}&0 \\
1&\frac{\lambda}{2}&k-\lambda-1&\frac{\lambda}{2}&0 \\
0&k&0&0&0 \\
\end{pmatrix} \quad L_4 = \begin{pmatrix}
0&0&0&0&1 \\
0&0&0&1&0 \\
0&0&1&0&0 \\
0&1&0&0&0 \\
1&0&0&0&0 \\
\end{pmatrix}$

\bigskip

$P =\begin{pmatrix}
1&k&2(v-k-1)&k&1 \\
1&\sqrt{  k }&0&-\sqrt{  k }&-1 \\
1&r&-2(r+1)&r&1 \\
1&-\sqrt{  k }&0&\sqrt{  k }&-1 \\
1&s&-2(s+1)&s&1 \\
\end{pmatrix} \quad Q =\begin{pmatrix}
1&\frac{v}{2}&f&\frac{v}{2}&g \\
1&\frac{v}{2\sqrt{k}}&\frac{rf}{k}&\frac{-v}{2\sqrt{k}}&\frac{sg}{k} \\
1&0&\frac{-(r+1)f}{v-k-1}&0&\frac{-(s+1)g}{v-k-1} \\
1&\frac{-v}{2\sqrt{k}}&\frac{rf}{k}&\frac{v}{2\sqrt{k}}&\frac{sg}{k} \\
1&-\frac{v}{2}&f&-\frac{v}{2}&g \\
\end{pmatrix}$

\bigskip

$L_1^* = \begin{pmatrix}0&\frac{(k-r)(k-s)}{2(k + rs)}&0&0&0\\
1&0&\frac{-(k+r)(k-s)(s+1)}{2(k+rs)(r-s)}&0&\frac{(k-r)(k+s)(r+1)}{2(k+rs)(r-s)}\\
0&\frac{(k^2-r^2)(k-s)} {4k(k+rs)}&0&\frac{(k-r)^2(k-s)} {4k(k+rs)}&0\\
0&0&\frac{-(k-r)(k-s)(s+1)}{2(k+rs)(r-s)}&0&\frac{(k-r)(k-s)(r+1)}{2(k+rs)(r-s)}\\
0&\frac{(k-r)(k^2-s^2)} {4k(k+rs)}&0&\frac{(k-r)(k-s)^2} {4k(k+rs)}&0\\
\end{pmatrix}$

\bigskip

$L_2^* = \begin{pmatrix}0&0&\frac{-k(k-s)(s+1)}{(k+rs)(r-s)}&0&0\\
0&\frac{-(k+r)(k-s)(s+1)}{2(k+rs)(r-s)}&0&\frac{-(k-r)(k-s)(s+1)}{2(k+rs)(r-s)}&0\\
1&0&\frac{(k-s)(ks^2-kr + 2ks - 2r^2s - r^2 + rs^2)}{(k+rs)(r-s)^2}&0&\frac{-(k-r)^2s(r+1)}{(k+rs)(r-s)^2}\\
0&\frac{-(k-r)(k-s)(s+1)}{2(k+rs)(r-s)}&0&\frac{-(k+r)(k-s)(s+1)}{2(k+rs)(r-s)}&0\\
0&0&\frac{(k-r)(k-s)s(s+1)}{(k+rs)(r-s)^2}&0&\frac{-(k-s)^2r(s+1)}{(k+rs)(r-s)^2}\\
\end{pmatrix}$

$L_3^* = \begin{pmatrix}0&0&0&\frac{(k-r)(k-s)}{2(k + rs)}&0\\
0&0&\frac{-(k-r)(k-s)(s+1)}{2(k+rs)(r-s)}&0&\frac{(k-r)(k-s)(r+1)}{2(k+rs)(r-s)}\\
0&\frac{(k-r)^2(k-s)} {4k(k+rs)}&0&\frac{(k^2-r^2)(k-s)} {4k(k+rs)}&0\\
1&0&\frac{-(k+r)(k-s)(s+1)}{2(k+rs)(r-s)}&0&\frac{(k-r)(k+s)(r+1)}{2(k+rs)(r-s)}\\
0&\frac{(k-r)(k-s)^2} {4k(k+rs)}&0&\frac{(k-r)(k^2-s^2)} {4k(k+rs)}&0\\
\end{pmatrix}$

\bigskip

$L_4^* = \begin{pmatrix}0&0&0&0&\frac{k(k-r)(r+1)}{(k+rs)(r-s)}\\
0&\frac{(k-r)(k+s)(r+1)}{2(k+rs)(r-s)}&0&\frac{(k-r)(k-s)(r+1)}{2(k+rs)(r-s)}&0\\
0&0&\frac{-(k-r)^2s(r+1)}{(k+rs)(r-s)^2}&0&\frac{(k-r)(k-s)r(r+1)}{(k+rs)(r-s)^2}\\
0&\frac{(k-r)(k-s)(r+1)}{2(k+rs)(r-s)}&0&\frac{(k-r)(k+s)(r+1)}{2(k+rs)(r-s)}&0\\
1&0&\frac{-(k-s)^2r(s+1)}{(k+rs)(r-s)^2}&0&\frac{(k-r)(kr^2 + 2kr - ks + r^2s - 2rs^2 - s^2)}{(k+rs)(r-s)^2}\\
\end{pmatrix}$

\bigskip

Below we give a table of parameters for 4-class double covers of strongly regular graphs with at most 100 vertices and with $\lambda_1=\lambda_2$. The table denotes existence and nonexistence of the putative 4-class double cover by `+',`-', with a `?' for cases where existence unknown to us. 
This is to match the conventions of the strongly regular graph table in \cite{SRG}.

\bigskip

\begin{tabular}{c|c|c||c|c|c}

 & SRG & Notes & & SRG & Notes\\
 \hline
! & [16,5,0,2] & Wells Graph & ? & [64,28,12,12] &  \\
- & [16,10,6,6] & Theorem \ref{linsys}  & - & [64,35,18,20] & Corollary \ref{lamnbhd}\\
- & [16,6,2,2] & Theorem \ref{linsys} & ? & [66,20,10,4] &  \\
3! & [16,9,4,6] & SmallGroup(1152,157478) & ? & [66,45,28,36] &  \\
? & [28,12,6,4] & only T(8) feasible & - & [96,19,2,4] & Corollary \ref{lamnbhd}  \\
- & [28,15,6,10] & Corollary \ref{lamnbhd} & ? & [96,76,60,60] &  \\
- & [36,10,4,2] & Corollary \ref{grid}  & ? & [96,20,4,4] &  \\
? & [36,25,16,20] &  & - & [96,75,58,60] & Corollary \ref{lamnbhd}  \\
- & [36,14,4,6] & Theorem \ref{linsys}  & - & [96,35,10,14] & Corollary \ref{lamnbhd} \\
? & [36,21,12,12] &  & ? & [96,60,38,36] & SRG unknown \\
- & [36,15,6,6] & Corollary \ref{lamnbhd} & - & [100,18,8,2] & Corollary \ref{grid} \\
? & [36,20,10,12] & 3900 feasible & ? & [100,81,64,72] &  \\
- & [40,12,2,4] & Theorem \ref{linsys} & - & [100,22,0,6] & Corollary \ref{HS} \\
- & [40,27,18,18] & Corollary \ref{lamnbhd} & ? & [100,77,60,56] &  \\
- & [56,10,0,2] & BCN p. 278 & - & [100,27,10,6] & Corollary \ref{lamnbhd} \\
? & [56,45,36,36] &  & ? & [100,72,50,56] &  \\
- & [64,14,6,2] & Corollary \ref{grid} & ? & [100,33,8,12] & SRG unknown \\
? & [64,49,36,42] &  & ? & [100,66,44,42] & SRG unknown \\
- & [64,18,2,6] & Theorem \ref{linsys} & ? & [100,36,14,12] &  \\
? & [64,45,32,30] &  & - & [100,63,38,42] & Corollary \ref{lamnbhd} \\
? & [64,21,8,6] &  & ? & [100,44,18,20] &  \\
? & [64,42,26,30] & no rank 3 & - & [100,55,30,30] & Corollary \ref{lamnbhd} \\
- & [64,27,10,12] & Corollary \ref{lamnbhd} & ? & [100,45,20,20] &  \\
? & [64,36,20,20] &  & ? & [100,54,28,30] &  

\end{tabular}

\bigskip

Notes: Details on the Wells graph can be found in \cite{BCN}. The 3 non-isomorphic 4-class double covers of the SRG-$(16,9,4,6)$ and were found by computer, by brute force using the solutions of the associated linear system, noting that one may take all the edges of a spanning tree to be weighted 0. The resulting solution space has dimension 14, thus the computation is not demanding. Only one of these is Schurian, from a transitive action of an order 1152 group which is small group number 157478 in magma's database. It is not clear if this example generalizes. 

For the SRG $(28,12,6,4)$, there are exactly 4 nonisomorphic strongly regular graphs, however only $T(8)$ gives a consistent linear system. For the SRG-$(36,20,10,12)$, there are exactly 32548 nonisomorphic graphs, but only 3900 have consistent linear systems. The SRG-$(64,42,26,30)$ are not classified, however the examples from rank 3 groups did not satisfy Theorem \ref{linsys}. 

The non-existence of the putative 4-class double cover of SRG-$(56,10,0,2)$ (the Gerwitz graph) is mentioned in \cite{BCN} on page 278. 

\section{ The case $\lambda_1 \neq \lambda_2$ }\label{ne}

In this section we deal with the case that $\lambda_1 \neq \lambda_2$.  We claim that, in this case, the eigenvalues must be integral, hence we must have $(\lambda_1-\lambda_2)^2+4k = u^2$ a perfect square.  We summarize this in the theorem below:

\begin{thm}
    For a 4-class double cover of a $(v,k,\lambda,\mu)$-strongly regular graph with $\lambda_1 \neq \lambda_2$, we have $(\lambda_1-\lambda_2)^2+4k$ is a perfect square. 
\end{thm}

\begin{proof}
    Let $1,m_1,m_2,m_3,m_4$ be the multiplicites of the scheme. From the $P$-matrix given in section 2.3, we have the system:

\bigskip

    $$1+m_1+m_2+m_3+m_4 = 2v$$

    $$k+\frac{( \lambda_1 - \lambda_2 )+\sqrt{ ( \lambda_1 - \lambda_2 )^2 +4k }}{2}m_1+rm_2+\frac{( \lambda_1 - \lambda_2 )-\sqrt{ ( \lambda_1 - \lambda_2 )^2 +4k }}{2} m_3 + sm_4 = 0$$

    $$-2(v-k-1)-2(r+1)m_2-2(s+1)m_4 = 0$$

    $$1-m_1+m_2-m_3+m_4 = 0$$    

\bigskip
    
    From the first and last equation we obtain $1+m_2+m_4=m_1+m_3=v$. The fact that $1+m_2+m_4=v$ together with equation 3 gives $k+rm_2+sm_4=0$.
    
    By way of contradiction, suppose $( \lambda_1 - \lambda_2 )^2 +4k$ is not a perfect square. Then the second equation forces  $m_1=m_3=\frac{v}{2}$. The second equation reduces to:     $k+( \lambda_1 - \lambda_2 )\frac{v}{2}+rm_2 + sm_4 = 0$, thus $( \lambda_1 - \lambda_2 )\frac{v}{2}= 0$. this forces $\lambda_1=\lambda_2$, contradicting our hypothesis.

\end{proof}

The parameters below are given in terms of $v,k,\lambda,\mu$, the nontrivial eigenvalues of the strongly regular graph $r,s$ of multiplicities $f,g$, the parameters $\lambda_1, \lambda_2,u$ and $l = \lambda_1-\lambda_2$.

\bigskip

$L_1 = \begin{pmatrix} 0&k&0&0&0 \\
1&\lambda_1&k-\lambda-1&\lambda_2&0 \\
0&\frac{\mu}{2}&k-\mu&\frac{\mu}{2}&0 \\
0&\lambda_2&k-\lambda-1&\lambda_1&1 \\
0&0&0&k&0 
\end{pmatrix} \quad L_2 = \begin{pmatrix}
0&0&2(v-k-1)&0&0 \\
0&k-\lambda-1&2(v-2k+\lambda)&k-\lambda-1&0 \\
1&k-\mu&2(v-2k+\mu-2)&k-\mu&1 \\
0&k-\lambda-1&2(v-2k+\lambda)&k-\lambda-1&0 \\
0&0&2(v-k-1)&0&0 \\
\end{pmatrix}$

\bigskip

$L_3 = \begin{pmatrix}
0&0&0&k&0 \\
0&\lambda_2&k-\lambda-1&\lambda_1&1 \\
0&\frac{\mu}{2}&k-\mu&\frac{\mu}{2}&0 \\
1&\lambda_1&k-\lambda-1&\lambda_2&0 \\
0&k&0&0&0 \\
\end{pmatrix} \quad L_4 = \begin{pmatrix}
0&0&0&0&1 \\
0&0&0&1&0 \\
0&0&1&0&0 \\
0&1&0&0&0 \\
1&0&0&0&0 \\
\end{pmatrix}$

\bigskip

$P =\begin{pmatrix}
1&k&2(v-k-1)&k&1 \\
1&\frac{u+l}{2}&0&\frac{-u-l}{2}&-1 \\
1&r&-2(r+1)&r&1 \\
1&\frac{-u+l}{2}&0&\frac{u-l}{2}&-1 \\
1&s&-2(s+1)&s&1 \\
\end{pmatrix} \quad Q =\begin{pmatrix}
1&\frac{(u-l)v}{2u}&f&\frac{(u+l)v}{2u}&g \\
1&\frac{v}{u}&\frac{-(k-s)r(s+1)}{(r-s)(k+rs)}&\frac{-v}{u}&\frac{(k-r)(r+1)s}{(r-s)(k+rs)} \\
1&0&\frac{-(k-s)}{r-s}&0&\frac{k-r}{r-s} \\
1&\frac{-v}{u}&\frac{-(k-s)r(s+1)}{(r-s)(k+rs)}&\frac{v}{u}&\frac{(k-r)(r+1)s}{(r-s)(k+rs)} \\
1&\frac{-(u-l)v}{2u}&\frac{-k(k-s)(s+1)}{(r-s)(k+rs)}&\frac{-(u+l)v}{2u}&\frac{k(k-r)(r+1)}{(r-s)(k+rs)} \\
\end{pmatrix}$

\bigskip

$L^*_1 = \begin{pmatrix}0&\frac{(u - l)(k-s)(k-r)}{2u(k+rs)}&0&0&0
\\1&0&\frac{-(s + 1)(u + l)(k-s)((u-l)^2 + 4r)}{8(r - s)u(k+rs)}&0&\frac{(r + 1)(u + l)(k-r)((u-l)^2 + 4s)}{8(r - s)u(k+rs)}
\\0&\frac{(k-s)(k-r)((u-l)^2 + 4r)}{4u^2(k+rs)}&0&\frac{(k-s)(k-r)^2}{u^2(k+rs)}&0
\\0&0&\frac{-(s + 1)(u - l)(k-s)(k-r)}{2(r - s)u(k+rs)}&0&\frac{(r + 1)(u - l)(k-s)(k-r)}{2(r - s)u(k+rs)}
\\0&\frac{(k-s)(k-r)((u-l)^2 + 4s)}{4u^2(k+rs)}&0&\frac{(k-s)^2(k-r)}{u^2(k+rs)}&0\\\end{pmatrix}$

\bigskip


$L^*_2 = \begin{pmatrix}0&0&\frac{-k(s + 1)(k-s)}{(r - s)(k+rs)}&0&0\\0&\frac{-(s + 1)(u + l)(k-s)((u-l)^2 + 4r)}{8(r - s)u(k+rs)}&0&\frac{-(s + 1)(u + l)(k-s)(k-r)}{2(r - s)u(k+rs)}&0\\1&0&\frac{(k-s)((k-r)(s^2+2s-r)-2r(s+1)(r-s))}{(r - s)^2(k+rs)}&0&\frac{-s(r + 1)(k-r)^2}{(r - s)^2(k+rs)}\\0&\frac{-(s + 1)(u - l)(k-s)(k-r)}{2(r - s)u(k+rs)}&0&\frac{-(s + 1)(u - l)(k-s)((u+l)^2 + 4r)}{8(r - s)u(k+rs)}&0\\0&0&\frac{s(s + 1)(k-s)(k-r)}{(r - s)^2(k+rs)}&0&\frac{-(s + 1)r(k-s)^2}{(r - s)^2(k+rs)}\\\end{pmatrix}$

\bigskip

$L^*_3 = \begin{pmatrix}0&0&0&\frac{(u + l)(k-s)(k-r)}{2u(k+rs)}&0\\
0&0&\frac{-(s + 1)(u + l)(k-s)(k-r)}{2(r - s)u(k+rs)}&0&\frac{(r + 1)(u + l)(k-s)(k-r)}{2(r - s)u(k+rs)}\\
0&\frac{(k-s)(k-r)^2}{u^2(k+rs)}&0&\frac{(k-s)(k-r)((u+l)^2 + 4r)}{4u^2(k+rs)}&0\\
1&0&\frac{-(s + 1)(u - l)(k-s)((u+l)^2 + 4r)}{8(r - s)u(k+rs)}&0&\frac{(r + 1)(u - l)(k-r)((u+l)^2 + 4s)}{8(r - s)u(k+rs)}\\
0&\frac{(k-s)^2(k-r)}{u^2(k+rs)}&0&\frac{(k-s)(k-r)((u+l)^2 + 4s)}{4u^2(k+rs)}&0\\\end{pmatrix}$

\bigskip

$L^*_4 = \begin{pmatrix}0&0&0&0&\frac{(r + 1)(u - l)(u + l)(k-r)}{4(r - s)(k+rs)}\\
0&\frac{(r + 1)(u + l)(k-r)((u-l)^2 + 4s)}{8(r - s)u(k+rs)}&0&\frac{(r + 1)(u + l)(k-s)(k-r)}{2(r - s)u(k+rs)}&0\\
0&0&\frac{-s(r + 1)(k-r)^2}{(r - s)^2(k+rs)}&0&\frac{r(r + 1)(k-s)(k-r)}{(r - s)^2(k+rs)}\\
0&\frac{(r + 1)(u - l)(k-s)(k-r)}{2(r - s)u(k+rs)}&0&\frac{(r + 1)(u - l)(k-r)((u+l)^2 + 4s)}{8(r - s)u(k+rs)}&0\\
1&0&\frac{-(s + 1)r(k-s)^2}{(r - s)^2(k+rs)}&0&\frac{(k-r)((k+s)(r^2+2r-s)-2rs(s+1))}{(r - s)^2(k+rs)}\\\end{pmatrix}$

\newpage

Below we give a table of parameters for 4-class double covers of strongly regular graphs with at most 100 vertices with $\lambda_1 \neq \lambda_2$. The table is similar to the previous one, except the additional parameters $\lambda_1,\lambda_2, l = \lambda_1-\lambda_2$ and $u$ are given. Since each 4-class double cover contains two double covers of the SRG $\Gamma$, namely $\hat{\Gamma}$ and $\check\Gamma$, we may interchange their roles obtain the same scheme with a different ordering on the relations, and with $\lambda_1$ and $\lambda_2$ swapped. Therefore without loss of generality, we always take $\lambda_1 \geq \lambda_2$, which we do in the table below.

\begin{tabular}{c|c|c|c}

 & SRG & $\lambda_1,\lambda_2,l,u]$ & Notes  \\
\hline
+  & [10, 6, 3, 4]  & [ 2, 1, 1, 5 ] & 	Section 2.4 \\
+  & [15, 8, 4, 4]  & [ 3, 1, 2, 6 ] & 	Section 2.4 \\
-  & [21, 10, 3, 6]  & [ 3, 0, 3, 7 ] & 	Theorem \ref{linsys} \\
+  & [21, 10, 5, 4]  & [ 4, 1, 3, 7 ] & 	Section 2.4 \\
+  & [28, 12, 6, 4]  & [ 5, 1, 4, 8 ] & 	Section 2.4 \\
!  & [35, 16, 6, 8]  & [ 6, 0, 6, 10 ] & 	$J(8,4)$ \\
+  & [36, 14, 7, 4]  & [ 6, 1, 5, 9 ] & 	Section 2.4 \\
+  & [36, 20, 10, 12]  & [ 9, 1, 8, 12 ] & 	$Q$-polynomial \\
?  & [40, 27, 18, 18]  & [ 12, 6, 6, 12 ] & 	all feasible \\
?  & [45, 32, 22, 24]  & [ 13, 9, 4, 12 ] & 	77 feasible \\
+  & [45, 16, 8, 4]  & [ 7, 1, 6, 10 ] & 	Section 2.4 \\
-  & [49, 12, 5, 2]  & [ 3, 2, 1, 7 ] & 	Cor \ref{grid} \\
+  & [55, 18, 9, 4]  & [ 8, 1, 7, 11 ] & 	Section 2.4 \\
+  & [55, 36, 21, 28]  & [ 15, 6, 9, 15 ] & 	From $M_{11}$ \\
?  & [56, 45, 36, 36]  & [ 20, 16, 4, 14 ] & 	Gerwitz, feasible  \\
?  & [57, 32, 16, 20]  & [ 10, 6, 4, 12 ] & 	 \\
?  & [63, 32, 16, 16]  & [ 10, 6, 4, 12 ] & 	 \\
+  & [63, 32, 16, 16]  & [ 15, 1, 14, 18 ] & 	$Q$-polynomial \\
?  & [64, 27, 10, 12]  & [ 8, 2, 6, 12 ] & 	 \\
!  & [64, 28, 12, 12]  & [ 12, 0, 12, 16 ] & 	halved cube \\
+  & [66, 20, 10, 4]  & [ 9, 1, 8, 12 ] & 	Section 2.4 \\
+  & [66, 45, 28, 36]  & [ 20, 8, 12, 18 ] & 	$Q$-polynomial  \\
+  & [78, 22, 11, 4]  & [ 10, 1, 9, 13 ] & 	Section 2.4 \\
-  & [81, 20, 1, 6]  & [ 1, 0, 1, 9 ] & 	Theorem \ref{Patv2} \\
?  & [81, 50, 31, 30]  & [ 18, 13, 5, 15 ] & 	 \\
?  & [85, 64, 48, 48]  & [ 30, 18, 12, 20 ] & 	 \\
?  & [85, 54, 33, 36]  & [ 18, 15, 3, 15 ] & 	SRG unknown \\
+  & [91, 24, 12, 4]  & [ 11, 1, 10, 14 ] & 	Section 2.4 \\
?  & [96, 75, 58, 60]  & [ 34, 24, 10, 20 ] & 	 \\
?  & [96, 35, 10, 14]  & [ 6, 4, 2, 12 ] & 	SRG unknown \\
?  & [96, 60, 38, 36]  & [ 21, 17, 4, 16 ] & 	SRG unknown \\
?  & [99, 56, 28, 36]  & [ 19, 9, 10, 18 ] & 	SRG unknown \\
?  & [100, 27, 10, 6]  & [ 8, 2, 6, 12 ] &  \\
-  & [100, 66, 41, 48]  & [ 30, 11, 19, 25 ] & 	$Q$-polynomial, \cite{GV}

\end{tabular}

\bigskip

Notes: In the cases where the strongly regular graphs are classified, all were checked by Theorem \ref{linsys} and the number that satisfy the theorem are listed. An example of a 4-class double cover of the unique SRG-$(55, 36, 21, 28)$ is given by the group $C_2 \times M_{11}$ acting on the cosets of a subgroup of $M_{11}$ of order 72, contained in the stabilizer of a point. The schemes that are $Q$-polynomial are previously known, see \cite{MMW}.

\section{Open Problems}

We give some open problems below.  

\begin{itemize}

    \item Complete the tables in this paper for all cases where the strongly regular graphs have been classified. 

    \item Classify the rank three strongly regular graphs that have 4-class double covers, and give the values of $\lambda_1,\lambda_2$. 

    \item Find the parameters of the nontrivial association schemes generated by double covers of the classical distance-regular graph families of unbounded diameter.
    
\end{itemize}

\section{Acknowledgements}

The programming languages Magma (see \cite{Magma}) and GAP \cite{GAP4} were used to gather data and compute the solution space for examples of the linear systems described in Section 3. The research of J. Williford was supported in part by the Simons Foundation Collaboration Grant 711898. 



\end{document}